\documentclass[letterpaper,10 pt,conference]{ieeeconf}  
\IEEEoverridecommandlockouts                              

\usepackage{comment}

\usepackage{subfigure}

\usepackage{xcolor}
\usepackage{xparse}
\usepackage{mathrsfs}
\usepackage{graphics} 
\usepackage{amsmath} 
\usepackage{amssymb}  
\usepackage{amsthm}
\usepackage{enumitem} 
\usepackage{hyperref}%
\hypersetup{colorlinks=true,  linkcolor=blue,  breaklinks=true,  urlcolor=blue,  citecolor=blue}

\usepackage{lipsum}
\usepackage{mwe}
\usepackage{soul}
\usepackage{url}

\usepackage[prependcaption, colorinlistoftodos]{todonotes}

\newcommand\oprocendsymbol{\hbox{$\triangle$}}
\newcommand\oprocend{\relax\ifmmode\else\unskip\hfill\fi\oprocendsymbol}

\renewcommand{\theenumi}{(\roman{enumi})}

\DeclareSymbolFont{bbold}{U}{bbold}{m}{n}
\DeclareSymbolFontAlphabet{\mathbbold}{bbold}
\newcommand{\vect}[1]{\mathbbold{#1}}

\NewDocumentCommand\norm{mg}{\lVert #1 \rVert \IfNoValueF{#2}{_{#2}}}
\NewDocumentCommand\bignorm{mg}{\big\lVert #1 \big\rVert \IfNoValueF{#2}{_{#2}}}
\NewDocumentCommand\Bignorm{mg}{\Big\lVert #1 \Big\rVert \IfNoValueF{#2}{_{#2}}}
\NewDocumentCommand\seminorm{mg}{\lvert\!\lvert\!\lvert #1 \rvert\!\rvert\!\rvert \IfNoValueF{#2}{_{#2}}}

\newtheorem{theorem}{Theorem}

\newtheorem{lemma}[theorem]{Lemma}

\newtheorem{proposition}[theorem]{Proposition}

\newcommand {\be}{\begin{equation}}
\newcommand {\ee}{\end{equation}}

\title{\LARGE\bf  On an extension of the Friedkin-Johnsen model: \\ The effects of a homophily-based influence matrix}
 
\author{Giorgia Disar\`o and Maria Elena Valcher
\thanks{G. Disar\`o and  M.E. Valcher are with the Dipartimento di Ingegneria dell'Informazione,
 Universit\`a di Padova,
    via Gradenigo 6B, 35131 Padova, Italy, e-mail:  \texttt{giorgia.disaro@phd.unipd.it, meme@dei.unipd.it}
    This paper has been submitted for possible presentation at the $62^{nd}$ IEEE Conference on Decision and Control, CDC'23.}
   }
  \date{}

\begin{document}
\maketitle

\begin{abstract}
In this paper we propose an extended version of the Friedkin-Johnsen (FJ) model 
that accounts for the effects of homophily mechanisms on the agents' mutual appraisals.
The proposed model consists of two difference equations.
The first  one describes the opinions' evolution, namely how agents modify their opinions taking into account both their personal beliefs and the influences of other agents, as in the standard FJ model.  Meanwhile, the second equation models how the 
 influence matrix involved in the opinion formation process  updates according to a homophily mechanism, by allowing both positive and negative appraisals. 
 We show that the proposed time-varying version of the classical FJ model always asymptotically converges to a constant solution.
   Moreover, in the case of a single discussion topic, the asymptotic behavior of the system  is derived in closed form. 
\end{abstract}


\section{Introduction} \label{intro}
During the last decades, understanding and describing the way we communicate and exchange ideas has been the focus of extensive investigation.
Opinion dynamics has become a very lively research field that attracts and combines  concepts and techniques from different disciplines, ranging from sociology, psychology and economy, to mathematics and control engineering. Such strong interest resulted in a large number of models trying to capture and mathematically formalize the process of opinion formation in a social network. 
 Opinion formation processes depend on a large number of variables, thus making it   difficult to create mathematical models that are sufficiently elementary  to be rigorously analysed and, at the same time, accurate enough to capture the complexity that characterizes social phenomena. 
Despite the clear simplifications that  the proposed models have introduced, they have been able to provide many insights into the dynamical processes of   diffusion and   evolution of opinions in human population \cite{ProskurnikovTempo_1,ProskurnikovTempo_2}. 
While the  initial interest focused mainly on models aimed at explaining consensus \cite{DeGroot}, more recently a lot of models have been proposed to justify observed behaviors of social groups such as disagreement, polarization and conflict \cite{AltafiniPlosOne,FJ_1990,HegselmannKrause,hiller}, which are even more frequent than consensus in real scenarios. 

Among them, one of the most famous is surely the Friedkin-Johnsen (FJ) model \cite{FJ_1990}, that captures the fact that the opinion of an individual 
on a topic evolves under the effects of two main driving forces. On the one hand, the individual (in the following also referred to as ``agent")  is influenced by the opinions on the same topic of his/her neighbours in the social network, each one weighted  by  the appraisal that the agent has of them. On the other hand,  agents  tend to   ``stick" to their own initial opinions (prejudices), that  therefore  
keep affecting   their opinions  at each subsequent time. This asymptotically leads  to opinions which are closer to each other than the initial opinions, but not identical, namely consensus is no longer reached. 
In the original FJ model \cite{FJ_1990} the opinions are expressed on a single topic and the influence matrix, that quantifies how much each agent values the opinions of the others, is constant and row stochastic.
 Later on, several extensions of the model have been proposed in the literature. 
 In particular, the model has been extended to the case of multiple topics \cite{multidim_opinions,FJ_extensions_conf,FJ_extensions}, with time varying row stochastic influence matrices \cite{time_varying_FJ}, and recently a version of the FJ model whose influence matrix has both positive and negative entries has been proposed \cite{Cinesi22}, thus accounting for the fact that relationships among individuals in a network may also be competitive/antagonistic (see \cite{Altafini2013,LP-HS-DL:21,Xia2016StructuralBA}).  
 
In all such models the influence matrix is either constant or time-varying, nonnegative or real valued, but it is always assumed to be independent of the dynamics of the agents' opinions. This assumption does not seem to be realistic since in real life very often the interpersonal relationships among the agents depend on the comparison of their opinions, following a homophily mechanism, namely the tendency of individuals to associate and interact more intensively with like-minded people
\cite{saturation,bin_hom,McPherson2001,HomophilyMei,Rivera2010}. In other words, agents tend to be influenced by individuals who hold similar opinions and, conversely, tend to give little or even negative weight to the opinions of agents with whom they mostly disagree.

In recent times, an interesting model of  the interplay between homophily-based appraisal dynamics and influence-based opinion dynamics has been proposed by F. Liu et al. \cite{MeiDorfler}. 
The model explores for the first time how the evolution of the opinions of a group of agents on a certain number of issues/topics is influenced by the agents' mutual appraisals and, conversely, the agents' mutual appraisals are updated based on the agents' opinions on the various issues, according to a homophily principle. More recently, 
 a simplified version of the model, that does not quantify the level of mutual appraisal but only its sign, has been proposed in \cite{EJC2022}. 
It has been shown that this model is simpler and yet equally accurate 
 in predicting the asymptotic evolution of the individuals' opinions in  small networks, as the ones we will consider in this paper. 
 
 In this contribution we  propose an extended version of the FJ model whose influence matrix is generated according to a homophily mechanism, by keeping into account only the  signs of the agents' appraisals. 
 In the general case, we have been able to prove that the opinion matrix of a group of $n$ agents on $m$ topics asymptotically converges to a constant solution, that strongly depends on the agents' initial opinions as well as on the agents' stubbornness coefficients, namely their attitudes to 
 remain attached  to their original opinions. 
 Finally, we consider the special   case where there is only one discussion topic and provide an explicit expression of the agents' asymptotic opinions.  
 
The paper is organized as follows: Section \ref{model} introduces the model explaining the meaning of all the quantities involved. Section \ref{general_case} provides the main results about the dynamics of the model. Section \ref{single_topic} addresses the single-topic case. In Section \ref{examples} two numerical examples are proposed. Finally, Section \ref{concl} concludes the paper.   
 \smallskip 
 

 {\bf Notation.}\ Given two integers $k$ and $n$, with $k \le n$, the symbol $[k,n]$ denotes the set $\{k, k+1, \dots, n\}$.
 We let  $\vect{0}_k$  denote the $k$-dimensional vector with all  zero entries,  and  $\vect{0}_{k\times l}$  the matrix of dimension $k\times l$ whose entries are all  zero. We denote by $\vect{e}_i$ the $i$-th canonical vector of dimension $n$, where $n$ is always clear from the context. 
 In the sequel, the $(i,j)$-th entry of a matrix  $A$ is denoted    by 
$[{A}]_{ij}$, 
while the $i$-th entry of a vector $v$   by $v_i$. 
The function ${\rm sgn}\!\!:\mathbb{R}^{n \times m}\rightarrow \{-1,0,1\}^{n \times m}$ is the   function that maps a real matrix $A$ into a matrix taking values in $\{-1,0,1\},$ in accordance with the sign of its entries, namely $[{\rm sgn} (A)]_{ij} = {\rm sgn}([A]_{ij})$ for every $i,j$. 
 The expression $X = {\rm blockdiag}\{X_1,\dots,X_k\}$ denotes the block diagonal matrix whose diagonal blocks are $X_1,\dots,X_k$. 
A {\em signature matrix} is a diagonal  matrix  whose diagonal entries belong to $\{-1,1\}$. 
The {\em infinity norm of a matrix} $A\in \mathbb{R}^{n \times n}$ is defined as $\Vert A \Vert_{\infty} := \max_{i\in[1,n]}{\sum_{k=1}^{n}|A_{ik}|}$. The {\em infinity norm of a vector} $v\in {\mathbb R}^n$ is $\|v\|_{\infty}  := \max_{i\in[1,n]} |v_i|$. \\
  The {\em spectrum} of a matrix $A$, denoted as $\sigma(A)$, is the set of all its eigenvalues, and the {\em spectral radius}, $\rho(A)$, is defined as $\rho(A) := \max \{|\lambda|: \lambda \in \sigma(A)\}$.
 
 In this paper by an {\em undirected   and signed graph} we mean  a triple $\mathcal{G}=(\mathcal{V},\mathcal{E},\mathcal{A})$, where $\mathcal{V}=[1,n]$ is the set of nodes (or vertices), $\mathcal{E}\subseteq \mathcal{V} \times \mathcal{V}$ is the set of edges (or arcs) and $\mathcal{A}\in \{-1,0,1\}^{n\times n}$ is the {\em adjacency matrix} of the graph $\mathcal{G}$. An arc $(j,i)\in \mathcal{E}$ if and only if $[\mathcal{A}]_{ij}\ne 0$. When so, $[{\mathcal A}]_{ij}$ represents the (positive or negative) weight of the arc. Moreover, due to the fact that the graph is undirected, the matrix $\mathcal{A}$ is symmetric and so $(i,j)\in \mathcal{E}$ if and only if $(j,i) \in \mathcal{E}$. 
 Since the adjacency matrix ${\mathcal A}$ uniquely identifies the graph, in the following we will use the notation ${\mathcal G}({\mathcal A})$.
 A graph $\mathcal{G}$ is said to be {\em structurally balanced}   \cite{Altafini2013,Xia2016StructuralBA} if the set of its nodes can be partitioned into two disjoint subsets such that (s.t.)  the weights of the edges between nodes belonging to the same subset are nonnegative, and the weights of the edges between nodes belonging to different subsets are nonpositive. 


\section{The model} \label{model}
Given a group of $n$ agents expressing their opinions on $m$ distinct  topics, we denote by $Y(t)\in\mathbb{R}^{n\times m}$ the \emph{opinion matrix at time t}, whose $(i,j)$-th entry represents the opinion that agent $i$ has about topic $j$ at   time   $t \in {\mathbb Z}_+$.\\
We denote by $W(t)\in\mathbb{R}^{n\times n}$ the \emph{influence matrix at time t}, whose $(i,j)$-th entry represents the influence that agent $j$ has on agent $i$ at   time   $t$. Specifically, we assume that: 
\begin{itemize}
\item $[W(t)]_{ij}>0$ $\Leftrightarrow$ $i$ positively regards the opinion of $j$;
\item $[W(t)]_{ij}<0$ $\Leftrightarrow$  $i$ negatively regards the opinion of $j$;
\item $[W(t)]_{ij}=0$ $\Leftrightarrow$  $i$ neglects the opinion of $j$.
\end{itemize}
We assume that  at every time $t$  the influence that agent $j$ has on agent $i$ is given by agent $i$'s appraisal of agent $j$. 
On the other hand, the appraisal that $i$ has of $j$ at time $t$  is based  on a homophily mechanism \cite{saturation,McPherson2001}, since  it depends on the comparison of the opinions that agents $i$ and $j$ have about all the topics at time $t-1$.  As in \cite{EJC2022}, we consider only the signs of the mutual appraisals, rather than their values. This is motivated by the fact that from a practical viewpoint it is complicated to quantify the appraisals each individual has of the others, but, on the contrary, it is easy to recognise if the relationship between  two agents is friendly or hostile. Moreover, this choice is   more robust to modelling errors and   more realistic, because agent $j$ can influence positively or negatively agent $i$'s opinion about a certain topic, but this influence does not necessarily scale with the absolute value of their mutual appraisal. Furthermore, we have chosen to account also for the fact that two agents decide not to rely on each other's opinions, i.e., $[W(t)]_{ij} = 0$. Indeed, in small-size networks, as the ones  we are considering, this corresponds to the case where agent $i$ knows agent $j$, but decides to neglect his/her opinions, for lack  of  correlation between their evaluations. Therefore, the fact that the mutual appraisal is zero is an information that should be considered, justifying the choice of dividing  each row of the influence matrix by $n$, instead of by the number of its non-zero entries. However,  it is worth noticing that condition $[W(t)]_{ij}=0$ is a very rare occurrence, as it will be clear in the following, since it corresponds to the case when 
the (real-valued) opinion vectors of agent $i$ and $j$ at time $t-1$ 
(i.e., the $i$-th and $j$-th rows of $Y(t-1)$)    are orthogonal.\\
 
Based on these premises,
in this paper we propose the following model,  representing the intertwining between an FJ-type opinion dynamics and a homophily-based appraisal mechanism: 
\begin{eqnarray}
\label{opinion_dyn}
Y(t+1) &= &(I_n-\Theta)W(t+1)Y(t)+\Theta Y(0), \\
\label{appraisal_dyn}
W(t+1) &=&\frac{1}{n}{\rm sgn}\left(Y(t) Y(t)^{\top}\right),
 \end{eqnarray}
where $\Theta\in {\mathbb R}^{n \times n}$ is a diagonal matrix. For every $i\in [1,n]$,  the nonnegative  diagonal entry $\theta_i$ of $\Theta$ represents the stubbornness
of agent $i$ in preserving the original opinion. In the paper we will steadily assume:
\smallskip

 {\em Assumption 1}. \ For every  $i\in [1,n]$ the stubbornness of agent $i$ satisfies $0 < \theta_i <1$.
\smallskip

  It is   easy to see that if the $i$-th row of $Y(0)$ is zero, then the $i$-th row of $Y(t)$ is zero for every $t\ge 0$. Similarly, if 
  the $i$-th column of $Y(0)$ is zero, then the $i$-th column of $Y(t)$ is zero for every $t\ge 0$. So, in the following we will rule out these cases, which are of no interest. 
\smallskip
  
  {\em Assumption 2}. \ The matrix $Y(0)\in {\mathbb R}^{n\times m}$ is devoid of zero rows and zero columns.
  \smallskip
  

  Finally, it is worth noticing that the influence matrix $W(t+1)$, as  defined, is a symmetric matrix for every $t\ge 0$.
  
%
%
 
\section{General results} \label{general_case}

In order to investigate the asymptotic behavior of the opinion matrix, we  first provide an alternative way to express the opinion matrix at time $t$, by introducing the transition matrix $M(t)$, relating 
$Y(t)$ to $Y(0)$. In the following we will steadily resort to the following notation:
\be
S_0  := Y(0) Y(0)^\top.
\label{S0}
\ee 

\begin{proposition}  \label{trans_mat}   
For every $Y(0)\in {\mathbb R}^{n \times m}$,  at every time $t\ge 0$, we have
\be
Y(t+1) = M(t+1)Y(0),
\label{Yt1}
\ee
where
\begin{eqnarray}
M(t+1) &=&    (I_n- \Theta) W(t+1) M(t) +\Theta,
\label{Mt1}\\
M(0) &=& I_n, \label{M0} \\
W(t+1) &=& \frac{1}{n}   {\rm sgn} (M(t)S_0M(t)^\top).
\label{Wt1}
\end{eqnarray}
\end{proposition}
\smallskip

\begin{proof} We prove the result by induction on $t$. We first show that the result is true for $t=0$.
We   observe that
$$W(1) = \frac{1}{n}   {\rm sgn} (Y(0) Y(0)^\top) = \frac{1}{n}   {\rm sgn} (M(0)S_0M(0)^\top),$$ and hence
\begin{eqnarray*}
Y(1) &=&  \left[ (I_n- \Theta)W(1) +\Theta\right] Y(0)  \\
&=&  \left[ (I_n- \Theta)W(1)M(0) +\Theta\right] Y(0)   = M(1) Y(0),
\end{eqnarray*}
 where 
 $$M(1)= (I_n- \Theta) W(1)M(0) +\Theta.$$
 Now we assume that equations \eqref{Yt1}, \eqref{Mt1} and \eqref{Wt1} are true for $t < \bar t$ and prove that they hold true also for $t=\bar t$. \\
 From 
 $$W(\bar t +1) = \frac{1}{n}{\rm sgn} (Y(\bar t)Y(\bar t)^\top)$$
 by the inductive assumption (on the expression of $Y$), we obtain
 $$W(\bar t +1) = \frac{1}{n}{\rm sgn} (M(\bar t)S_0M(\bar t)^\top).$$
 On the other hand,
\begin{eqnarray*}
Y(\bar t +1) &=&  (I_n - \Theta) W(\bar t+1) Y(\bar t) + \Theta Y(0)\\
&=& \left[(I_n- \Theta) W(\bar t+1) M(\bar t)   +\Theta\right] Y(0) \\
&=& M(\bar t+1) Y(0),
\end{eqnarray*}
where 
$M(\bar t+1) = (I_n- \Theta) W(\bar t+1) M(\bar t)   +\Theta.$
\end{proof}
\medskip

  Based on Proposition \ref{trans_mat}, we   now   derive the main result regarding the asymptotic behavior of the 
  sequence $\{M(t)\}_{t\in {\mathbb Z}_+}$.
\medskip

\begin{theorem} \label{M_bounded}
For every $Y(0)\in {\mathbb R}^{n \times m}$, the solution of the system in \eqref{Mt1}-\eqref{M0}-\eqref{Wt1} is bounded, namely\footnote{Note that $W(t+1), t\in {\mathbb Z}_+,$ is always bounded, since it takes values in $\{-1/n,0,1/n\}$.}
 $\exists M\in\mathbb{R}_+$ s.t. $ \Vert M(t) \Vert _{\infty} \leq M $ for all $t\geq 0$. 
Moreover, there exists  $M_\infty := \lim_{t\to +\infty} M(t)$ and it satisfies 
\be
M_\infty =    (I_n- \Theta) \frac{1}{n} {\rm sgn} (M_\infty S_0 M_\infty^\top) M_\infty  +\Theta, 
\label{asymptM}
\ee
i.e., it is an equilibrium point of    \eqref{Mt1} for $W(t+1)$ expressed as in \eqref{Wt1}.
\end{theorem}

\begin{proof}
The solution of system \eqref{Mt1}, with initial condition  \eqref{M0} and  $W(t+1)$  as in \eqref{Wt1}, can be expressed as the sum  of the following two (unforced and forced) terms:
$$M(t) = M_u(t)+M_f(t),\qquad t\ge 1,$$
where 
{\small
\begin{eqnarray*} 
M_u(t) \!\!\!\!&=&\!\!\!\! (I-\Theta)W(t)(I-\Theta)W(t-1)\dots (I-\Theta)W(1)M(0)\\
M_f(t) \!\!\!\!&=&\!\!\!\!  \Big[  I + (I-\Theta)W(t)+(I-\Theta)W(t)(I-\Theta)W(t-1) \\
\!\!\!&&\!\!\! \ \ \ +\dots+(I-\Theta)W(t)\dots(I-\Theta)W(2)\Big]\Theta,
\end{eqnarray*}}
 which becomes $M_f(t)=\Theta$ for $t=1$.
Therefore $\forall t\ge 1$
$$\Vert M(t) \Vert_{\infty} \le \Vert M_u(t) \Vert_{\infty} + \Vert M_f(t) \Vert_{\infty}.$$
We first observe that $\forall k \geq 1$ 
\begin{eqnarray*}
\Vert (I-\Theta)W(k)\Vert_{\infty} &=& \max_{i}{\left\{(1-\theta_i)\sum_{j=1}^{n}{\left|\left[W(k)\right]_{ij}\right|}\right\}} \\
&\leq & \max_{i}{(1-\theta_i)} =: \alpha <1.
\end{eqnarray*}
Therefore, for every $t\geq 1$, we have (recall that $M(0)=I_n$)
\begin{eqnarray*}
\Vert M_u(t) \Vert _{\infty} &=& \Vert (I-\Theta)W(t)(I-\Theta)W(t-1)\dots \\
& & (I-\Theta)W(1)\Vert_{\infty} \\
&\leq & \prod_{k=1}^{t}{\Vert (I-\Theta)W(k)\Vert_{\infty}} \le \alpha^t.
\end{eqnarray*}
Similarly, for every $t\ge 1$
\begin{eqnarray*}
\Vert M_f(t) \Vert _{\infty} \!\!&\leq &\!\!  \Vert I+(I-\Theta)W(t) +\\
\!\!& &\!\! +(I-\Theta)W(t)(I-\Theta)W(t-1)+\dots  \\
\!\!& &\!\! +(I-\Theta)W(t)\dots(I-\Theta)W(2)\Vert_{\infty} \Vert \Theta \Vert_{\infty} \\
\!\!&\leq &\!\! [\Vert I \Vert_{\infty}+\Vert (I-\Theta)W(t)\Vert_{\infty} +\\
\!\!& &\!\! +\Vert (I-\Theta)W(t)(I-\Theta)W(t-1)\Vert_{\infty}+\dots  \\
\!\!& &\!\! + \Vert (I-\Theta)W(t)\dots(I-\Theta)W(2)\Vert_{\infty}] \Vert \Theta\Vert_{\infty} \\
\!\!&\leq &\!\! \left[1+\alpha+\alpha^2+\dots+\alpha^{t-1}\right] \Vert \Theta\Vert_{\infty} \\
\!\!&=&\!\! \frac{1-\alpha^t}{1-\alpha} \Vert \Theta\Vert_{\infty}.
\end{eqnarray*}
Therefore, for every $t\ge 1$, 
$$\Vert M(t) \Vert_{\infty} \le \alpha +  \Vert \Theta\Vert_{\infty}.$$
This  shows that $M(t)$ is bounded $\forall t \geq 0$.  \medskip

To prove that there exists  $M_\infty = \lim_{t\to +\infty} M(t)$, we observe that 
$$\lim_{t \to +\infty} \Vert M_u(t)\Vert_\infty \le \lim_{t \to +\infty} \alpha^t =0.$$ This ensures that 
$\lim_{t \to +\infty}  M_u(t)=\vect{0}_{n\times n}$ and hence $\lim_{t \to +\infty}  M(t) = \lim_{t \to +\infty}   M_f(t).$ 
We now observe that
\begin{equation*}\begin{split}
M_f(t) &= \Theta + \sum_{k=0}^{t-2} \Big[(I-\Theta)W(t)(I-\Theta)W(t-1)\dots  \\
& \ \ \  (I-\Theta)W(t-k)\Big] \Theta,
\end{split}\end{equation*}
and hence
\begin{equation*}\begin{split}
\lim_{t \to +\infty}   M_f(t) &= \Theta + \lim_{t \to +\infty} \sum_{k=0}^{t-2} \Big[(I-\Theta)W(t) \\ 
& \ \ \ (I-\Theta)W(t-1)\dots(I-\Theta)W(t-k)\Big]\Theta.
\end{split}\end{equation*}
So, we are remained with proving that the series in the previous expression converges.
It is well know   (see, e.g., Theorem 3.45 in \cite{math_analysis}, which   easily extends to series of matrices) that if the series of the norms converges, i.e., 
$\lim_{t \to +\infty} \sum_{k=0}^{t-2} \Vert (I-\Theta)W(t)(I-\Theta)W(t-1)\dots  (I-\Theta)W(t-k)\Vert < +\infty$, then the series
$\lim_{t \to +\infty} \sum_{k=0}^{t-2}   (I-\Theta)W(t)(I-\Theta)W(t-1)\dots  (I-\Theta)W(t-k)$ converges, in turn.
But, as a result of the previous analysis, we can claim that 
\begin{equation*}\begin{split}
&  \lim_{t \to +\infty} \sum_{k=0}^{t-2} \Vert (I-\Theta)W(t)(I-\Theta)W(t-1)\dots  \\
& \ \ \ \ \ \ \ \ \ \ \ \ \ \ (I-\Theta)W(t-k)\Vert \\
&\le \lim_{t \to +\infty} \sum_{k=0}^{t-2} \alpha^{k+1} = \lim_{t \to +\infty} \alpha\frac{1-\alpha^{t-1}}{1-\alpha} = \frac{\alpha}{1-\alpha} < +\infty.
\end{split}\end{equation*}
Finally, we have 
\begin{equation*}\begin{split}
\lim_{t\to\infty}{M(t+1)} &=  \lim_{t\to\infty}{\Big[(I-\Theta)\frac{1}{n}{\rm sgn}\left(M(t) S_0M(t) ^\top\right) \cdot} \\
& \ \ \ \ \ \ \ \ \ \ \ M(t)+\Theta\Big],
\end{split}\end{equation*}
which implies that $M_\infty$ satisfies \eqref{asymptM}.
\\ This completes the proof.  
\end{proof}
\smallskip

 The main consequence of Theorem \ref{M_bounded} is that for every $Y(0)\in {\mathbb R}^{n\times m}$ 
$$Y_\infty := \lim_{t\to +\infty} Y(t)$$
exists and coincides with $M_\infty Y(0)$. Therefore, the proposed extended version of the FJ model asymptotically converges to a constant 
solution. 
Moreover, there exists 
$W_\infty := \lim_{t\to +\infty} W(t+1)$ and  
\be
W_\infty \!=\! \frac{1}{n}{\rm sgn}\left(M_\infty S_0M_\infty ^\top\right)\in \! \left\{-\frac{1}{n},0,\frac{1}{n}\right\}^{n \times n} \!.
\label{finally}
\ee
 As a consequence of    Assumption 2,  the diagonal elements of $W_\infty$ are all positive and thus equal to $\frac{1}{n}$. 
\medskip

\begin{lemma}\label{Winf_diagonal}
For every $Y(0)\in {\mathbb R}^{n\times m}$ we have $$[W_\infty]_{ii}=\frac{1}{n}, \ \ \ \forall i\in[1,n],$$ and 
either $W_\infty$ has an eigenvalue in $1$ (and if so, ${\mathcal G}(W_\infty)$  is  structurally balanced) or is Schur stable.
\end{lemma}

\begin{proof}
First of all, it follows from \eqref{appraisal_dyn} that 
\be
W_\infty = \frac{1}{n} {\rm sgn}\left(Y_\infty Y_\infty^\top\right).
\ee Therefore, for every $i\in[1,n]$, we have that $n[W_\infty]_{ii} = {\rm sgn}(\vect{e}_i^\top Y_\infty Y_\infty^\top \vect{e}_i)={\rm sgn}(\Vert Y_\infty^\top \vect{e}_i\Vert^2)$, which can be either equal to $0$ or equal to $1$. 
We suppose, by contradiction, that $\exists i \in[1,n]$ s.t. $[W_\infty]_{ii}=0=n[W_\infty]_{ii}$ and we assume, without loss of generality, that 
$${\mathcal I}:= \left\{i\in[1,n] : [W_\infty]_{ii}=0\right\} = [1,n-k], \ \exists k\in[1,n-1].$$
Thus, for every $i \in[1,n-k]$, we have that $\Vert Y_\infty^\top \vect{e}_i\Vert^2=0$ and this is true if and only if $Y_\infty^\top \vect{e}_i = \vect{0}_n$. Therefore, we also have that $\vect{e}_i^\top Y_\infty Y_\infty^\top \vect{e}_j = 0,\ \forall j \in[1,n]$. It follows that $W_\infty$ can be block-partitioned in this way:
$$ W_\infty = \left[\begin{array}{c|c}
\vect{0}_{(n-k)\times (n-k)} & \vect{0}_{(n-k)\times k}  \\
\hline
\vect{0}_{k\times (n-k)} & W_{nz}
\end{array}\right],
$$
where $W_{nz}\in\left\{-\frac{1}{n},0,\frac{1}{n}\right\}^{k\times k}$ and $[W_{nz}]_{ii} = \frac{1}{n}, \ \forall i \in [1,k]$.
Now, we accordingly partition also the matrices $M_\infty$ and $\Theta$, obtaining 
$$ M_\infty = \left[\begin{array}{c|c}
M_{11} & M_{12} \\
\hline
M_{21}&M_{22}
\end{array}\right], \ \ \ 
\Theta = \left[\begin{array}{c|c}
\Theta_1 &  \\
\hline
&\Theta_2
\end{array}\right].
$$
From  \eqref{asymptM} and \eqref{finally}, it follows that 
\begin{eqnarray*}
\left[\begin{array}{c|c}
M_{11} & M_{12} \\
\hline
M_{21}&M_{22}
\end{array}\right] 
\!\!\!\!\!\!&=&\!\!\!\!\!\! \left[\begin{array}{c|c}
\Theta_1 &  \vect{0}_{(n-k)\times k} \\
\hline
(I_k-\Theta_2)W_{nz}M_{21}&(I_k-\Theta_2)W_{nz}M_{22}
\end{array}\right].
\end{eqnarray*}
This is equivalent to
$$
\begin{cases}
M_{11} = \Theta_1 \\
M_{12} = \vect{0}_{(n-k)\times k} \\
M_{21} = (I_k-\Theta_2)W_{nz}M_{21} \\
M_{22} = (I_k-\Theta_2)W_{nz}M_{22}+\Theta_2
\end{cases}
$$
and, in particular, the third equation implies that 
$$\left[(I_k-\Theta_2)W_{nz}\right]M_{21}\vect{e}_i = M_{21}\vect{e}_i, \ \forall i \in[1,n-k],$$
which means that $M_{21}\vect{e}_i$ is an eigenvector of $(I_k-\Theta_2)W_{nz}$ corresponding  to the eigenvalue $\lambda = 1$. On the other hand, it can be easily shown that $(I_k-\Theta_2)W_{nz}$ is Schur stable and so $M_{21}\vect{e}_i$ must be equal to $\vect{0}_k$. The same holds for every $i \in[1,n-k]$, leading to $M_{21} = \vect{0}_{k\times (n-k)}$.
This means that $M_\infty$ has the following block diagonal structure
$$M_\infty = \left[\begin{array}{c|c}
\Theta_1 &  \vect{0}_{(n-k)\times k} \\
\hline
 \vect{0}_{k\times (n-k)}&(I_k-\Theta_2)W_{nz}M_{22}
\end{array}\right].
$$ 
Moreover, if we partition $S_0$ as $
 \left[\begin{array}{c|c}
S_{11} & S_{12} \\
\hline
S_{12}^\top&S_{22}
\end{array}\right] 
$,  equation \eqref{finally} becomes
\begin{eqnarray*}
W_{\infty} \!\!\!\!\!&=&\!\!\!\!\! \left[\begin{array}{c|c}
\vect{0}_{(n-k)\times (n-k)} & \vect{0}_{(n-k)\times k}  \\
\hline
\vect{0}_{k\times (n-k)} & W_{nz}
\end{array}\right] \\
\!\!\!\!\!&=&\!\!\!\!\! \frac{1}{n} {\rm sgn} \left(
\left[\begin{array}{c|c}
\Theta_1 &  \\
\hline
&M_{22}
\end{array}\right]
\left[\begin{array}{c|c}
S_{11} & S_{12} \\
\hline
S_{12}^\top&S_{22}
\end{array}\right]
\left[\begin{array}{c|c}
\Theta_1 &  \\
\hline
&M_{22}^\top 
\end{array}\right] \right) \\
\!\!\!\!\!&=&\!\!\!\!\! \frac{1}{n} {\rm sgn} \left(
\left[\begin{array}{c|c}
\Theta_1 S_{11} \Theta_1 & \Theta_1 S_{12}M_{22}^\top \\
\hline
M_{22} S_{12}^\top \Theta_1 & M_{22}S_{22}M_{22}^\top
\end{array}\right] \right).
\end{eqnarray*}
But, the first diagonal block $\Theta_1S_{11}\Theta_1$ cannot be zero since this would imply (by Assumption 1)   $S_{11}=\vect{0}_{(n-k)\times (n-k)}$, which is not possible since $Y(0)$ has no zero rows by Assumption 2. \\
Therefore, we can conclude that $[W_\infty]_{ii} = \frac{1}{n}, \ \forall i \in[1,n]$. 
Finally, applying  Lemma 17 in \cite{EJC2022}, we can show that either $W_\infty$ has an eigenvalue in $1$ (and if so, ${\mathcal G}(W_\infty)$  is  structurally balanced) or is Schur stable.
\end{proof}

 It is possible to prove that
if $W_\infty$ has no zero entries, then  the influence matrix converges to its limit value $W_\infty$  in a finite number of steps. 
\medskip

\begin{proposition}\label{finite_time_Wconv}
Assume that  $W_\infty = \lim_{t\to +\infty}{W(t)}$ is devoid of zero entries.
Then $$\exists \ T\geq 0 \ \text{s.t.}   \ W(t)=W_\infty, \ \forall t \geq T.$$
\end{proposition}

\begin{proof} Let $i$ and $j$ be arbitrary indices in $[1,n]$, 
and
assume, for instance, that $[W_\infty]_{ij} = 1/n$. This means that
$\lim_{t\to +\infty}{\rm sgn} ([Y(t)Y(t)^\top]_{ij}) =1$ and hence $\lim_{t\to +\infty} [Y(t)Y(t)^\top]_{ij}  = [Y_{\infty}]_{ij}> 0.$
This implies that   $\exists \ T\geq 0$ s.t. for every $t \ge T$ we have $[Y(t)Y(t)^\top]_{ij} >0$,  and hence 
$[W(t+1)]_{ij} = [W_\infty]_{ij} = 1/n$  for every $t \ge T$.  
\end{proof}

\medskip

  We now  explore some interesting properties
of  
$M_\infty$.

 \begin{proposition}\label{diag_max_and_pos} For every $Y(0)\in {\mathbb R}^{n\times m}$, the matrix
$M_\infty = \lim_{t\to +\infty}  M(t)$ is nonsingular and such that 
\medskip

\begin{enumerate} [label=(\roman*)]
\item $\|M_\infty \vect{e}_i\|_\infty = \max_{j\in[1,n]}{|[M_\infty]_{ji}|}=|[M_\infty]_{ii}|,  \forall i \in [1,n]$;
\item $[M_\infty]_{ii}>0, \ \ \forall i \in [1,n]$.
\end{enumerate}
 \end{proposition}
 
 \begin{proof} We first  prove that $M_\infty$ is nonsingular. Suppose, by contradiction, that $v\in\mathbb{R}^n, v\ne \vect{0}_n,$ belongs to the kernel of $M_\infty$, i.e., $M_\infty v=\vect{0}_n$. Then, by making use of \eqref{asymptM} and \eqref{finally}, we obtain
$$\vect{0}_n=M_\infty v=(I-\Theta)W_\infty M_\infty v + \Theta v \ \Rightarrow \ \Theta v = \vect{0}_n,$$
which is not possible since each $\theta_i \in (0,1)$, by  Assumption 1. 

{\em (i)} 
 Let $i$ be any index  in $[1,n]$.
Then
 $$M_\infty \vect{e}_i =(I-\Theta)W_\infty M_\infty \vect{e}_i+\Theta \vect{e}_i.$$
 If we permute the entries of $M_\infty \vect{e}_i$, using an $n \times n$ permutation matrix $P$, in such a way that $$\tilde{v} :=P^\top M_\infty \vect{e}_i=\begin{bmatrix}
\tilde v_1 \\
 \vdots \\ 
 \tilde v_n
 \end{bmatrix},
 \ \ \text{with } |\tilde v_1| \geq |\tilde v_2| \geq \dots \geq |\tilde v_n|,$$ 
 we obtain 
 \begin{eqnarray*}
 \tilde{v}=P^\top M_\infty \vect{e}_i  &=& P^\top(I-\Theta)PP^\top W_\infty PP^\top M_\infty \vect{e}_i + \\
 &+& P^\top \Theta PP^\top \vect{e}_i \\
 &=& (I-\tilde{\Theta})\tilde{W} _\infty \tilde{v} + \tilde{\Theta} \vect{e}_j,  \qquad  \exists j\in [1,n],
 \end{eqnarray*}
 where $\tilde{W} _\infty = P^\top W_\infty P$ and $$\tilde{\Theta} = P^\top\Theta P = \begin{bmatrix}
 \tilde{\theta} _1 & & \\
 &\ddots & \\
 & & \tilde{\theta} _n
 \end{bmatrix}.$$  
 By looking at the first component of $\tilde{v}$, i.e., $\tilde v_1$, we have
 $$\tilde v_1 = (1-\tilde{\theta} _1) \vect{e}_1 ^\top \tilde{W} _\infty \begin{bmatrix}
\tilde v_1 \\ 
 \vdots \\ 
\tilde v_n
 \end{bmatrix}+\tilde{\theta} _1 \vect{e}_1^\top \vect{e}_j,$$
 which implies that 
 \begin{equation}
 |\tilde v_1|\leq  (1-\tilde{\theta} _1) \sum_{i=1}^{n}{\frac{|\tilde v_i|}{n}}+\tilde{\theta} _1 \vect{e}_1^\top \vect{e}_j.
 \label{useful}
 \end{equation}
 Therefore, if $j\ne 1$, the right-hand side of \eqref{useful} would be
 $$(1-\tilde{\theta} _1) \sum_{i=1}^{n}{\frac{|\tilde v_i|}{n}} < \sum_{i=1}^{n}{\frac{|\tilde v_i|}{n}}\leq |\tilde v_1|,$$
 a contradiction. Thus, it must be $j=1$ and $\tilde{\theta} _1 = \theta _i$. 
 So, we have $\tilde v_1=\left[M_\infty \vect{e}_i\right]_i = [M_\infty]_{ii}$. This means that  $\max_{j\in[1,n]}{|[M_\infty]_{ji}|}=|[M_\infty]_{ii}|$. 
 Clearly, this is true for every index $i \in [1,n]$, namely for every column of $M_\infty$.  
 
{\em (ii)} We want to prove that $[M_\infty]_{ii}>0,  \forall i \in [1,n]$, which is equivalent to showing that $\tilde v_1>0$, by referring to the notation adopted in  part {\em (i)}.  Suppose, by contradiction, that $\tilde v_1\leq 0$. Then, using the fact that 
 $[\tilde W_\infty]_{11}=\frac{1}{n}$,  we get
 \begin{eqnarray*}
\tilde  v_1 &=& (1-\tilde{\theta} _1) \vect{e}_1 ^\top \tilde{W} _\infty \begin{bmatrix}
 \tilde v_1 \\ 
 \vdots \\ 
\tilde  v_n
 \end{bmatrix}+\tilde{\theta} _1 \\
 &=& (1-\tilde{\theta} _1) \frac{1}{n}\tilde  v_1 +  (1-\tilde{\theta} _1)  \sum_{j \ne 1}{[\tilde{W} _\infty] _{1j}\tilde  v_j} +\tilde{\theta} _1.
 \end{eqnarray*}
Consequently,
$$\left(1-\frac{1-\tilde{\theta} _1}{n}\right)\tilde  v_1 -\tilde{\theta} _1 =  (1-\tilde{\theta} _1)   \sum_{j \ne 1}{[\tilde{W} _\infty]_{1j}\tilde  v_j}.$$
If $\tilde v_1\leq 0$, we get 
\begin{eqnarray*}
  \left(1-\frac{1-\tilde{\theta} _1}{n}\right)|\tilde v_1| +\tilde{\theta} _1 
&=& \left| \left(1-\frac{1-\tilde{\theta} _1}{n}\right)\tilde v_1 -\tilde{\theta} _1 \right| \\
&=& (1-\tilde{\theta} _1)   \left| \sum_{j \ne 1}{[\tilde{W} _\infty]_{1j}\tilde v_j} \right| \\
&\leq & (1-\tilde{\theta} _1)\frac{n-1}{n}|\tilde v_1|,
\end{eqnarray*}
which implies that 
$$ \left(1-\frac{1-\tilde{\theta} _1}{n}-\frac{1-\tilde{\theta} _1}{n}(n-1)\right)|\tilde v_1|+\tilde{\theta} _1 \leq 0 $$
$$\Rightarrow \ \ \left(1-(1-\tilde{\theta} _1)\right)|\tilde v_1|+\tilde{\theta} _1 \leq 0  \ \ \Rightarrow \ \ \tilde{\theta} _1|v_1|+\tilde{\theta} _1\leq 0,$$
a contradiction. \\
Therefore,  $\tilde v_1$ must be positive, 
 which is equivalent to saying that  $[M_\infty]_{ii}>0, \forall i \in [1,n]$, as we wanted to show. 
 \end{proof}
 \smallskip

 The previous result means  that the initial opinion of each agent impacts  more on his/her own final opinion than on
 the final opinions of the other agents. In other words, the agent that weights more agent $i$'s initial opinion is agent $i$ himself/herself. 
 Moreover, (and not unexpectedly!) such impact is always   positive.
  \medskip

\section{Single-topic case} \label{single_topic}

We now address  the case where $m=1$, namely there is only one discussion topic. When so, the opinion matrix is a column vector, that we now denote by $y(t)\in\mathbb{R}^n$, containing the opinions of the agents on the topic.
It is easy to see that if we define
$$v(t) := {\rm sgn}(y(t)),$$
   the influence matrix becomes 
\begin{displaymath}
W(t+1)=\frac{1}{n}{\rm sgn}\left(y(t)y(t)^{\top}\right) = \frac{1}{n} v(t)v(t)^{\top}.
\end{displaymath}
Consequently,
model \eqref{opinion_dyn}-\eqref{appraisal_dyn}
 becomes: 
\begin{eqnarray}
\label{op_dyn_1}
y(t+1) &=& (I-\Theta)W(t+1)y(t)+\Theta y(0) \\
\label{app_dyn_1}
W(t+1)&=&\frac{1}{n}{\rm sgn}\left(y(t)y(t)^{\top}\right) = \frac{1}{n}v(t)v(t)^{\top},
 \end{eqnarray}
leading to the  difference equation:
\begin{equation}
\label{op_dyn_ST}
y(t+1) = \frac{1}{n}(I-\Theta)v(t)v(t)^{\top}y(t)+\Theta y(0).
\end{equation}
 We also note that in this context Assumption 2  amounts to imposing that $y(0)$ is devoid of zero entries. 
 In fact, condition $y_i(0) = 0$ would lead the $i$-th agent to remain isolated and stick to the zero opinion.
\smallskip

Under the previous hypotheses, we can  derive the following results.
\medskip

\begin{lemma}
\label{W_evolution}
For $m=1$, the influence matrix 
 remains constant 
 and, specifically, 
$$W(t+1)=W(1) 
=\frac{1}{n}v(0)v(0)^{\top}, \quad \forall t\geq 1.$$
\end{lemma}

\begin{proof}
By induction on $t$.  For $t=1$, we have 
$$W(t+1)=W(2)=\frac{1}{n}{\rm sgn}(y(1)y^{\top}(1))=\frac{1}{n}v(1)v^{\top}(1).$$
But, 
\begin{eqnarray*}
v(1)&=&{\rm sgn}(y(1))\\
&=&{\rm sgn}\left[(I-\Theta)\frac{1}{n}v(0) v(0)^{\top}y(0)+\Theta y(0)\right] = v(0),
\end{eqnarray*}
where we exploited the fact that 
$v(0)^{\top}y(0) = {\rm sgn} (y(0))^\top y(0) > 0$ (Assumption 2 rules out the case $y(0)=\vect{0}_n$).
Suppose that the result holds for $t < \bar t$. For $t= \bar t$: 
\begin{eqnarray*}
v(\bar t+1) &=& {\rm sgn}(y(\bar t+1)) \\
&=&{\rm sgn}\left[(I-\Theta)\frac{1}{n}\overbrace{v(\bar t)}^{=v(0)}\underbrace{v(\bar t)^{\top}y(t)}_{\sum_{i=1}^{n}{|y_i(\bar t)|}>0}+\Theta y(0)\right] \\
&=& v(0).
 \end{eqnarray*}
Thus, $v(t+1)=v(0), \forall t\geq 0,$ 
yielding 
$W(t+1) =\frac{1}{n}v(t+1)v(t+1)^{\top}=\frac{1}{n}v(0)v(0)^{\top}=W(1), \forall t\geq 1$.
\end{proof} 
\smallskip

As a consequence of the previous lemma, for $m=1$
the model in \eqref{op_dyn_1}-\eqref{app_dyn_1} becomes time-invariant and the dynamics of $y(t)$ can be expressed as: 
\begin{equation}
\label{op_dyn_ST_Wconstant}
y(t+1) = \frac{1}{n}(I-\Theta)v(0)v(0)^{\top}y(t)+\Theta y(0).
\end{equation}

Lemma \ref{W_evolution} implies
that the whole opinion dynamics evolves at each time step 
with an influence matrix that corresponds to a situation
of  structural balance \cite{Altafini2013,Xia2016StructuralBA}, by this meaning that ${\mathcal G}(W(t+1))$ is structurally balanced for every $t\ge 0$.
%
We can now derive the following result.
\medskip

\begin{theorem}
\label{y_evolution} 
For $m=1$,  we have
\begin{eqnarray}
W_\infty &=& \frac{1}{n} v(0) v(0)^\top, \nonumber\\
 M_\infty  &=&
  \left[I_n +  \frac{1}{ \sum_{i=1}^n  \theta_i}   (I_n-\Theta)  v(0) v(0)^\top \right]\Theta.
  \label{minfm1}
\end{eqnarray}
\end{theorem}

\begin{proof}
By Lemma \ref{W_evolution}, $W(t)= \frac{1}{n}v(0)v(0)^{\top},  \forall t\ge 0$. Therefore, we also have $W_\infty = \frac{1}{n} v(0) v(0)^\top$. \\
For what concerns $M_\infty$, it follows from \eqref{asymptM} and \eqref{finally} that 
$$ [I_n-(I_n-\Theta)W_\infty]M_\infty =  \Theta, $$
which can be rewritten as
$$M_\infty = [I_n-(I_n-\Theta)W_\infty]^{-1} \Theta,$$ 
since the matrix $ I_n-(I_n- \Theta)  W_\infty$ is invertible, as shown below. 
By Gershgorin Circles Theorem \cite{HornJohnson} (and    Assumption 1), the spectrum of $(I_n- \Theta)  W_\infty$ satisfies
\smallskip

$\sigma((I_n- \Theta)  W_\infty) \subseteq$
\begin{eqnarray*}
&\subseteq&
\bigcup_{i=1}^n \left\{s\in {\mathbb C}: \left| s - \frac{1-\theta_i}{n}\right| \le \frac{n-1}{n} (1-\theta_i)\right\}\\
&=& \bigcup_{i=1}^n \left\{s\in {\mathbb C}: - \frac{n-2}{n}(1-\theta_i) \le  s  \le 1-\theta_i\right\}\\
&\subsetneq& \{s\in {\mathbb C}: |s| < 1\},
\end{eqnarray*}
where we used the fact that $[W_\infty]_{ii} = \frac{1}{n}$ for every $i\in [1,n]$,  and $|[W_\infty]_{ij}| \le \frac{1}{n}$ for every $i\ne j$, with  $i,j \in [1,n]$.
This implies that $(I_n- \Theta)  W_\infty$ is Schur stable. Therefore, $\rho((I_n- \Theta)  W_\infty)<1$. 
This in turn implies that $I-(I_n- \Theta)  W_\infty$ is an invertible matrix.  Moreover, 
it holds that 
\begin{eqnarray*}
[I_n-(I_n-\Theta)W_\infty]^{-1} \!\!\!&=&\!\!\! I_n + \sum_{k=1}^{+\infty}{[(I_n-\Theta)W_\infty]^k} \\
\!\!\!&=&\!\!\! I_n + \sum_{k=1}^{+\infty}{[(I_n-\Theta)\frac{1}{n} v(0) v(0)^\top]^k} \\
\!\!\!&=&\!\!\! I_n + [(I_n-\Theta)\frac{1}{n} v(0) v(0)^\top] \cdot\\
\!\!\!&&\!\!\! \ \ \ \  \ \ \ \sum_{k=1}^{+\infty}{\left( \frac{\sum_{i=1}^{n}{(1-\theta_i)}}{n}\right)^{k-1}} \\
\!\!\!&=&\!\!\! I_n +  \frac{1}{ \sum_{i=1}^n  \theta_i}   (I_n-\Theta)  v(0) v(0)^\top. 
\end{eqnarray*}
Thus,  $M_\infty$ is expressed as in \eqref{minfm1}.
\end{proof}

To conclude, we can provide an explicit expression for the agents' asymptotic opinions, namely 
$$
y_\infty = \left[I_n +  \frac{1}{ \sum_{i=1}^n  \theta_i}   (I_n-\Theta)  v(0) v(0)^\top \right]\Theta y(0).
$$

\section{Examples}\label{examples}


{\bf Example 1.} \ We consider a group of $n= 6$ agents discussing $m=6$ topics.
We assume that $\theta_1=\theta_6= \frac{2}{3}, \theta_2=\theta_5=	\frac{1}{2}, \theta_3=	\theta_4=\frac{1}{3}$ and that $Y(0)$ is: 
\begin{footnotesize}
$$Y(0)=\begin{bmatrix} 
-0.1317 \!\!&\!\!  1.7035 \!\!&\!\!  -0.2350  \!\!&\!\!  0.0802  \!\!&\!\!  0.7824  \!\!&\!\! -0.6380\cr
   0.2968  \!\!&\!\! -0.6272  \!\!&\!\!  0.9015 \!\!&\!\!  -0.4425 \!\!&\!\!  -0.1206 \!\!&\!\!  -0.7040\cr
  -0.6075 \!\!&\!\!  -0.3453 \!\!&\!\!   0.3935 \!\!&\!\!  -0.9496  \!\!&\!\!  0.5671  \!\!&\!\! -0.3654\cr
   0.5217  \!\!&\!\! -0.2691  \!\!&\!\! -0.2884  \!\!&\!\! -0.1193 \!\!&\!\!  -0.3721 \!\!&\!\!  -1.1914\cr
   0.0244  \!\!&\!\! -0.2168  \!\!&\!\! -0.2278  \!\!&\!\!  1.1211 \!\!&\!\!  -0.3104  \!\!&\!\! -0.7398\cr
  -0.3392  \!\!&\!\!  0.7993  \!\!&\!\!  0.1429  \!\!&\!\! -0.9816 \!\!&\!\!  -1.4906  \!\!&\!\!  0.2002\end{bmatrix}.$$
  \end{footnotesize}
The evolutions of the opinions on the $6$ topics as well as the evolution of the influence matrix are illustrated in Figure \ref{figura1}.

\begin{figure}

 \begin{center}
         \includegraphics[width=7cm]{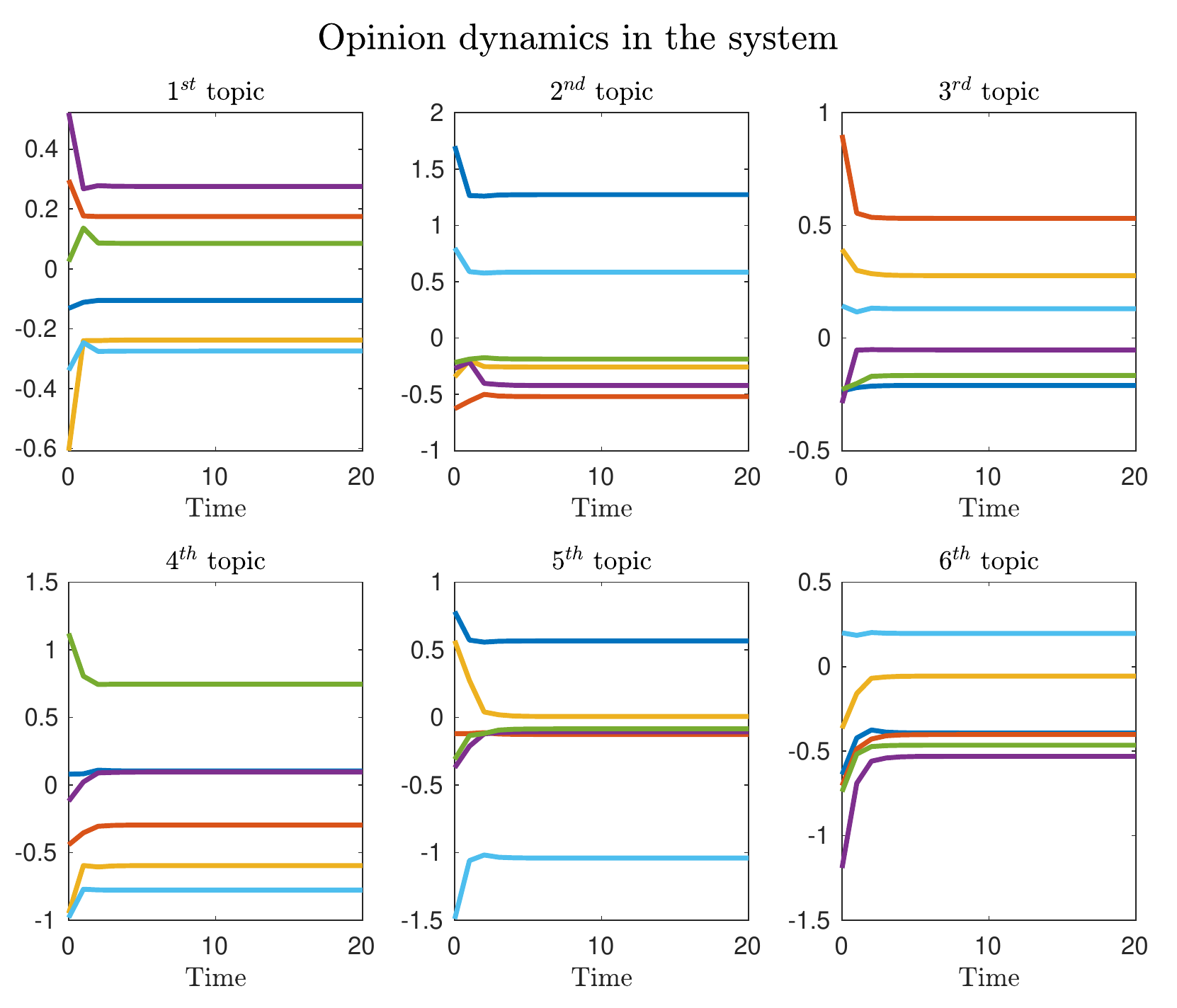}

  \smallskip
  
         \includegraphics[width=7cm]{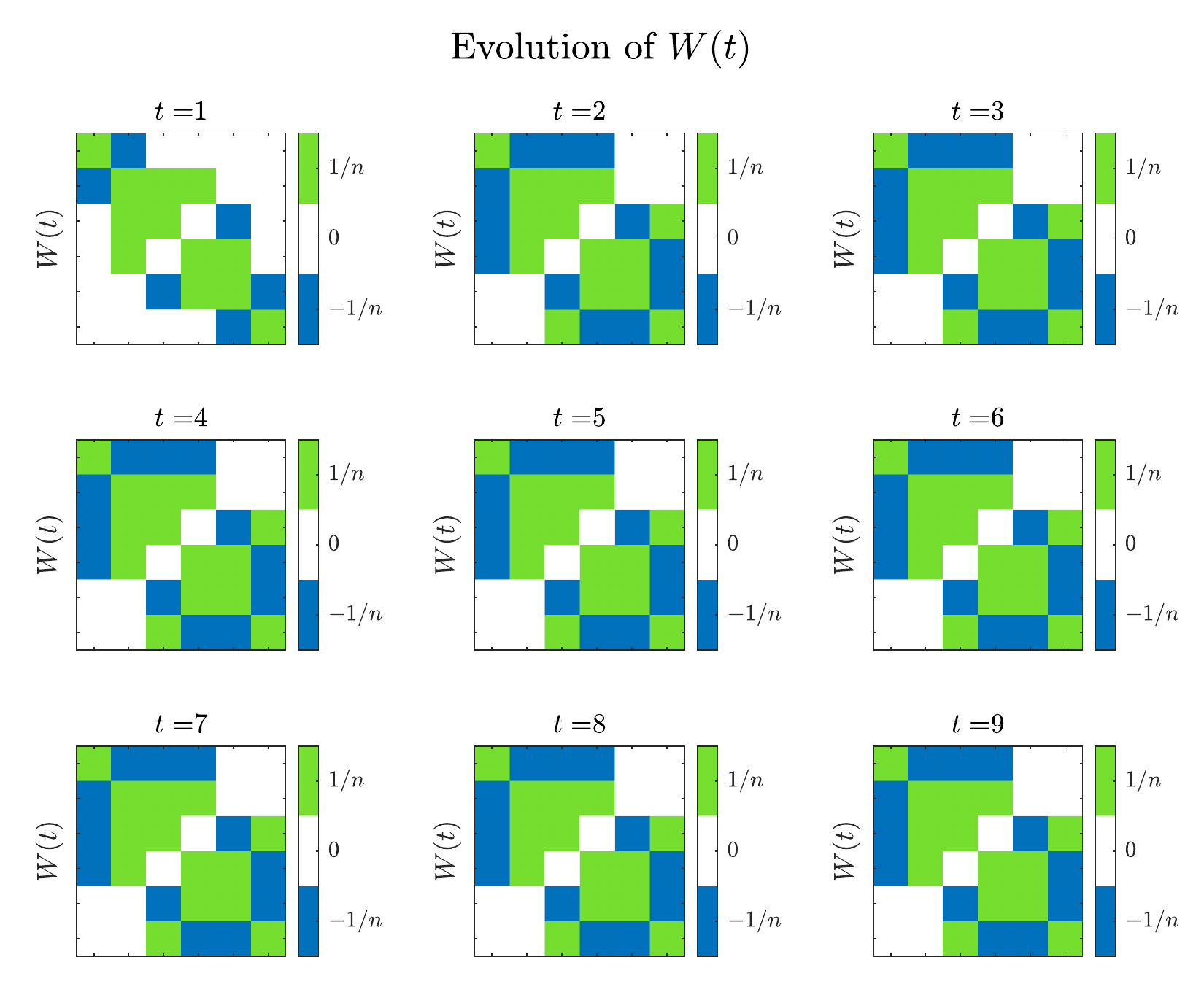}    
    \smallskip      
 
         \includegraphics[width=7cm]{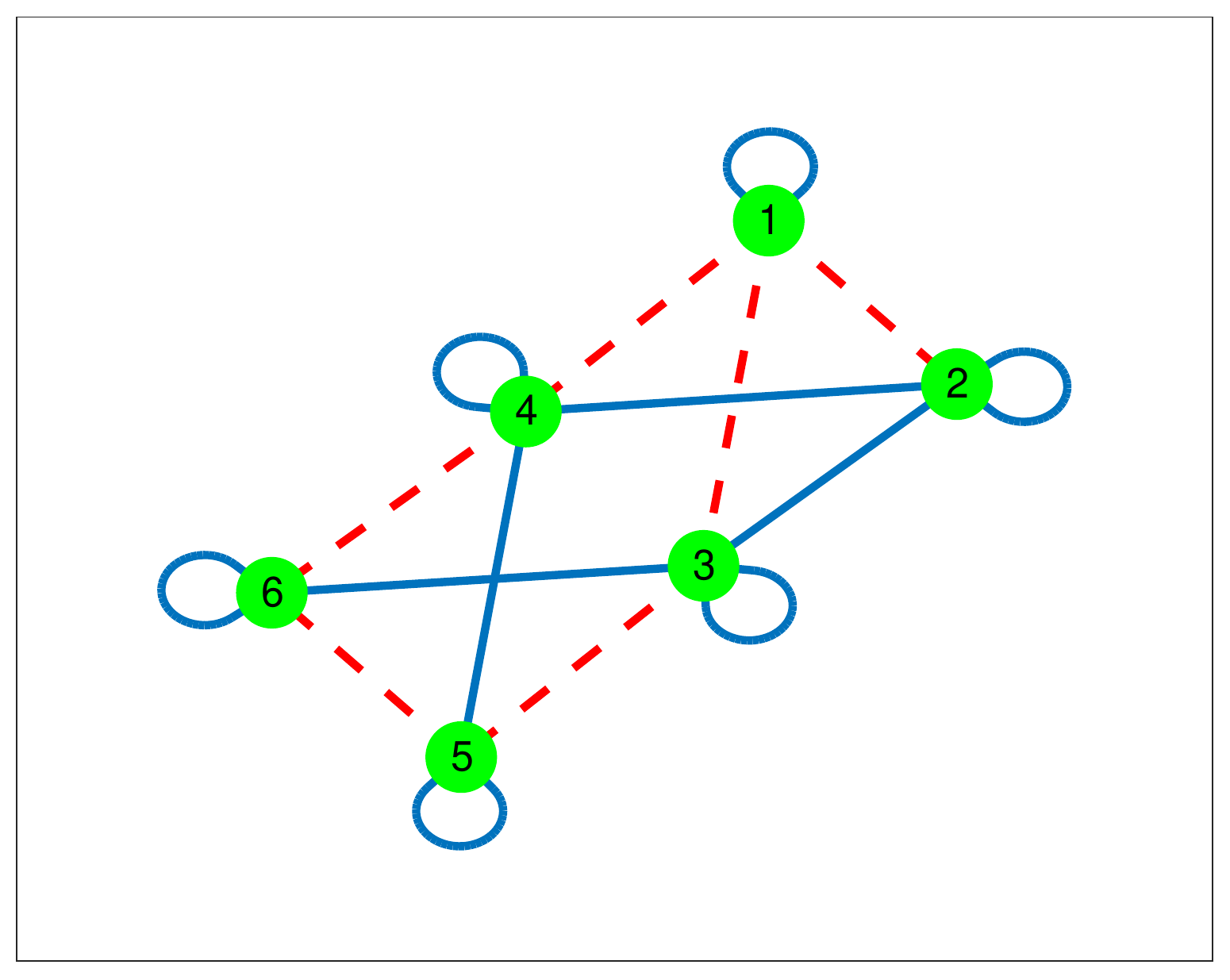}

     \caption{Example 1: Evolutions of the opinions on the $6$ topics (top);\\ Evolution of the influence matrix (middle); Graph associated with $W_\infty$ (bottom).} \label{figura1}

\end{center}    
\end{figure}

\medskip

{\bf Example 2.}\ We consider a group of $n= 5$ agents discussing $m=3$ topics.  We assume that $\theta_1=0.1174,  \theta_2=0.2967, \theta_3=	0.3188, \theta_4=	0.4242, \theta_5=0.5079,$ and that $Y(0)$ is:  

\begin{footnotesize}
$$Y(0)= \begin{bmatrix} 
-18.8898 &  47.9748  &  9.4896\cr
  42.3380 &  -6.1130 & -23.7788\cr
  -6.9793  &-38.8881 &  10.2843\cr
 -31.5184 & -24.1935 &  21.1216\cr
  40.4881  & -9.1280 & -27.8253\end{bmatrix}.$$
  \end{footnotesize}
 The evolutions of the opinions on the $3$ topics as well as the evolution of the influence matrix are illustrated in Figure \ref{figura2}.


 \begin{figure}

  \begin{center}     
         \includegraphics[width=7cm]{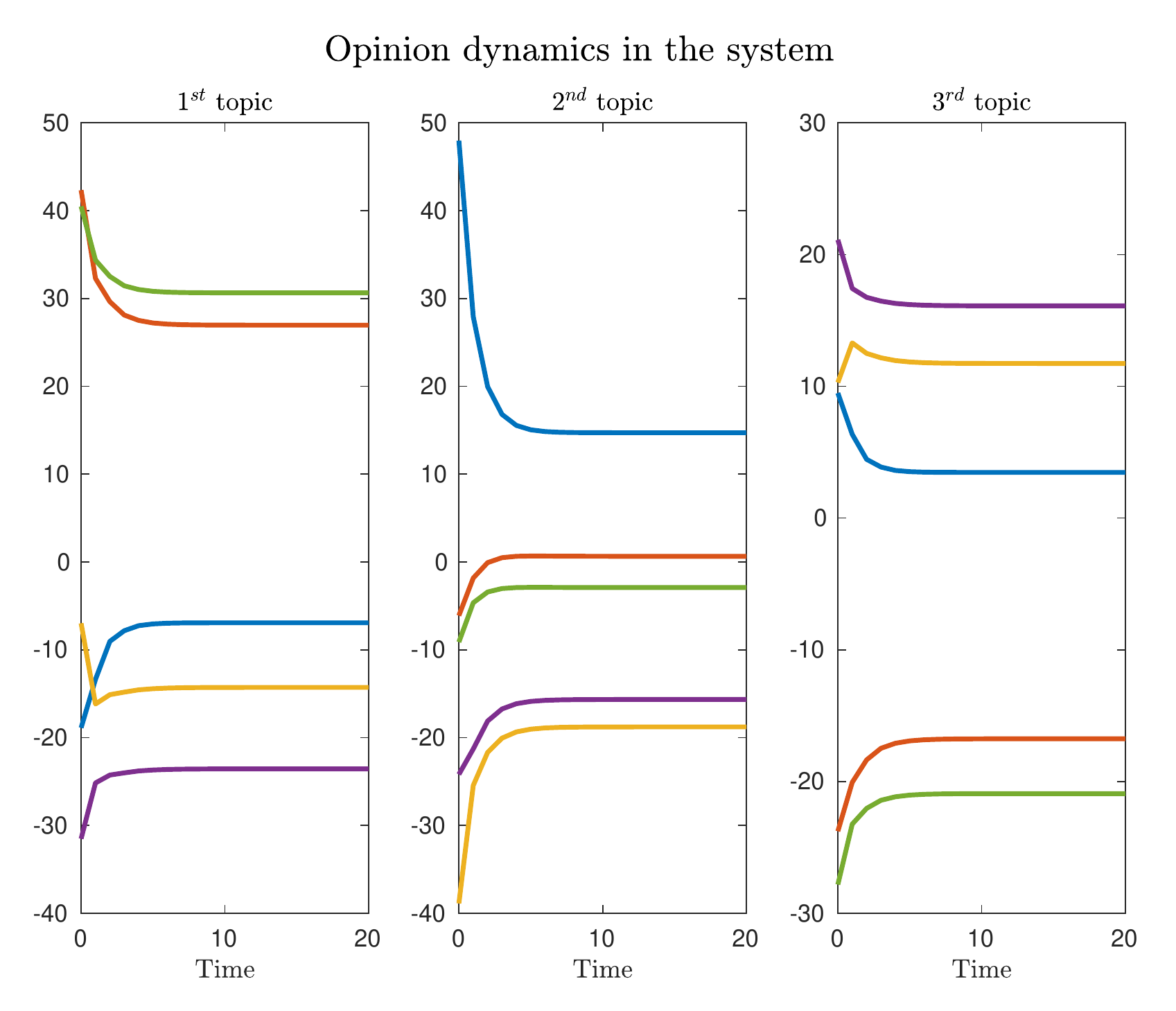}

  \smallskip
  
         \includegraphics[width=7cm]{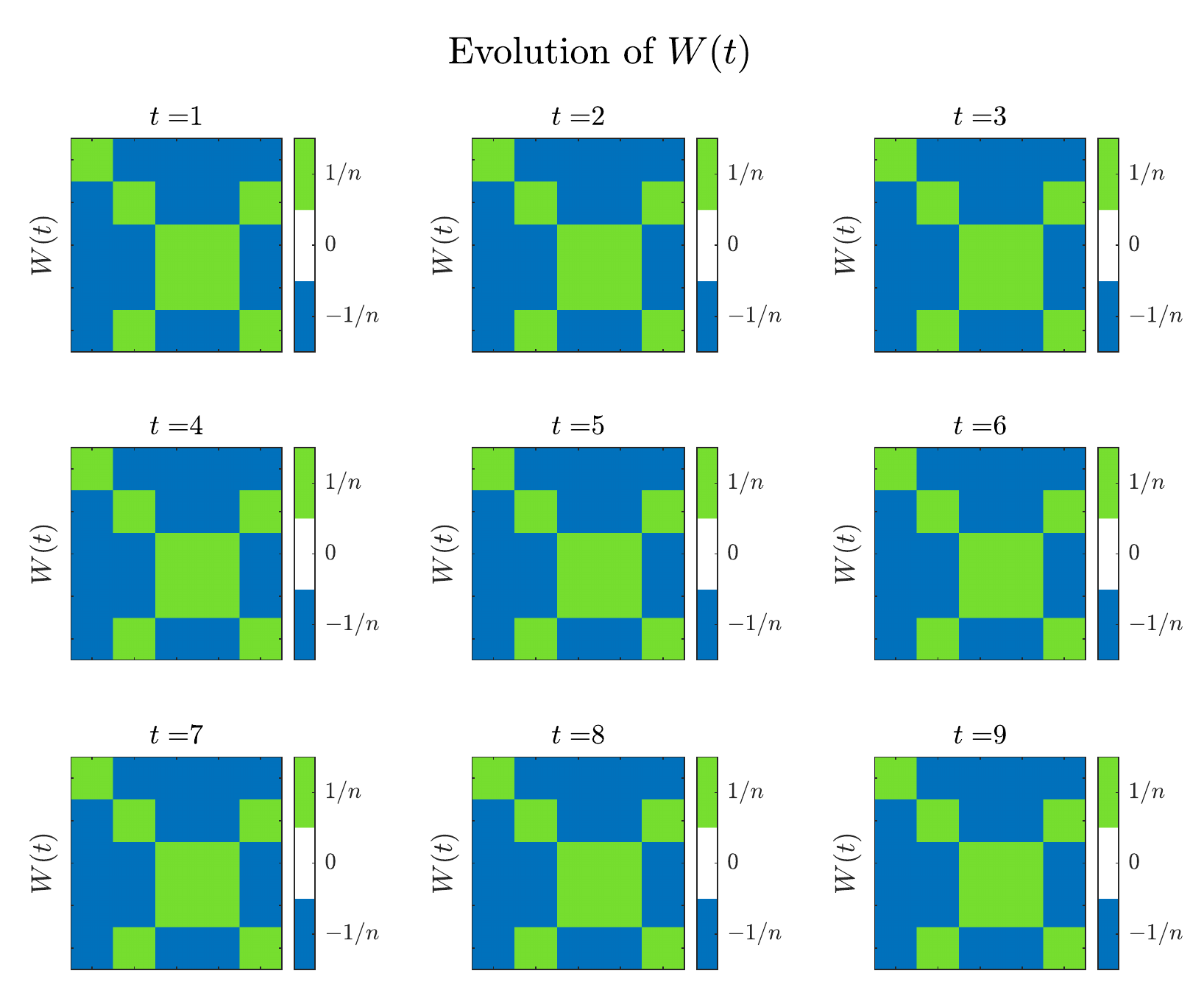}    
         
 \smallskip
 
         \includegraphics[width=7cm]{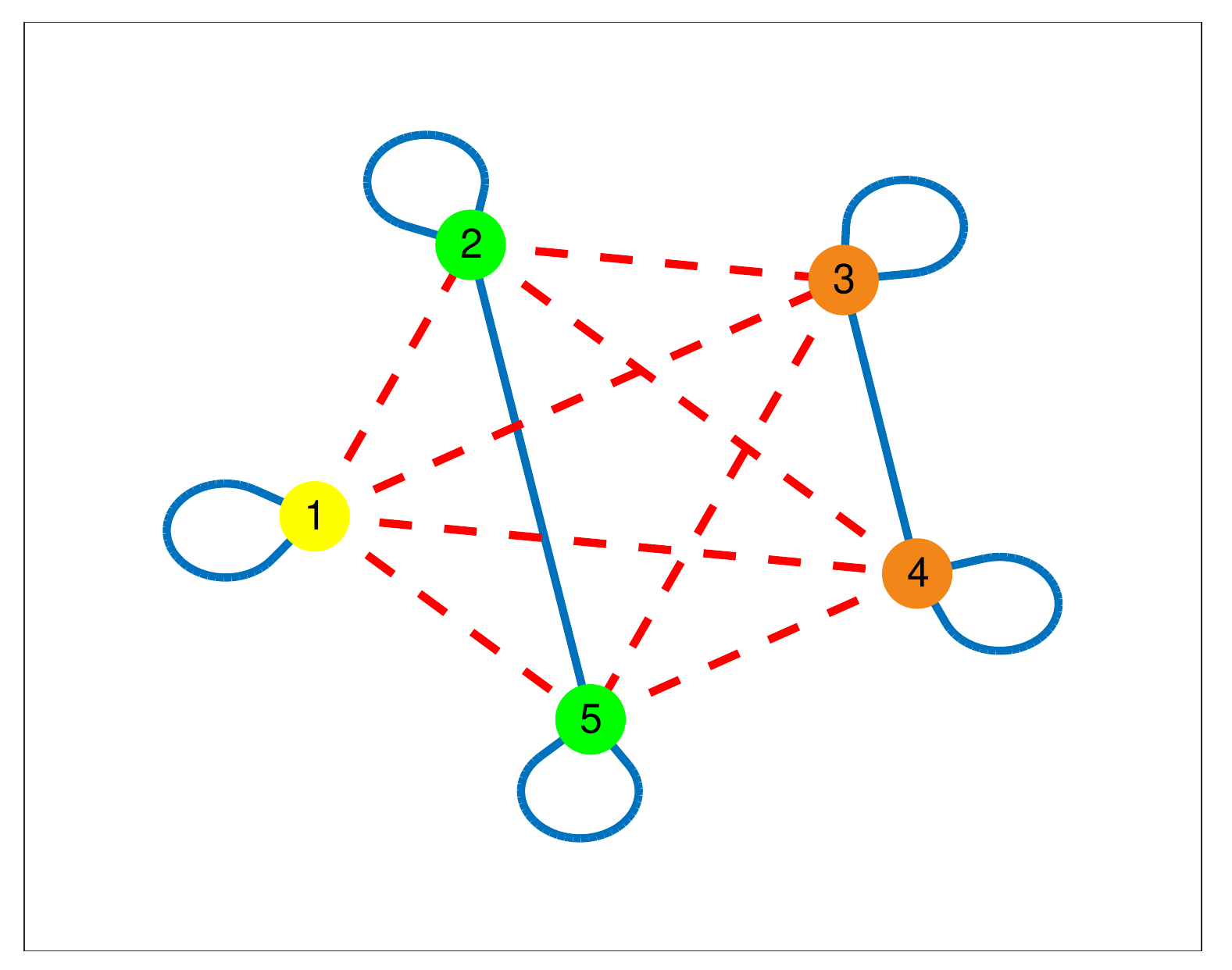}


     \caption{Example 2: Evolutions of the opinions on the $3$ topics (top);\\ Evolution of the influence matrix (middle); Graph associated with $W_\infty$ (bottom). }\label{figura2}
 \end{center} 
\end{figure}
%
 

 
\section{Conclusions} \label{concl}
In this paper we  proposed an extended version of the Friedkin-Johnsen model with a time-varying influence matrix,
 accounting for cooperative and competitive interactions among individuals. In particular, we  
 assumed that the influence matrix updates based on 
a homophily mechanism. We   proved that the agents' opinions asymptotically converge even if the structure of the network varies over time and the individuals are not all friendly with each other. Moreover, we   also highlighted some interesting properties of the limit values of the  transition matrix  and of the influence matrix.
Finally, in the  special case where $m=1$, namely there is a single discussion topic,
 the asymptotic behavior of the opinion vector was derived in closed form.



\begin{thebibliography}{100}

\bibitem{AltafiniPlosOne}
C.~Altafini.
\newblock Dynamics of opinion forming in structurally balanced social networks.
\newblock {\em Plos One}, 7 (6):e38135, 2012.

\bibitem{Altafini2013}
C.~Altafini.
\newblock Consensus problems on networks with antagonistic interactions.
\newblock {\em IEEE Trans. Aut. Contr.}, 58 (4):935--946, 2013.

\bibitem{saturation}
P.~Dandekar, A.~Goel, and D.~T. Lee.
\newblock Biased assimilation, homophily, and the dynamics of polarization.
\newblock {\em Proceedings of the National Academy of Sciences},
  110(15):5791--5796, 2013.

\bibitem{bin_hom}
G.~De~Pasquale and M.~E. Valcher.
\newblock A binary homophily model for opinion dynamics.
\newblock In {\em Proceedings of the European Control Conference (ECC) 2021},
  pages 1057--1087, Rotterdam, The Nederlands, 2021.

\bibitem{EJC2022}
G.~{De Pasquale} and M.E. Valcher.
\newblock A bandwagon bias based model for opinion dynamics: Intertwining
  between homophily and influence mechanisms.
\newblock {\em European Journal of Control}, 68, file 100675, 2022.

\bibitem{DeGroot}
M.~H. DeGroot.
\newblock Reaching a consensus.
\newblock {\em Journal of the American Statistical Association},
  69(345):118--121, 1974.

\bibitem{multidim_opinions}
N.E. Friedkin.
\newblock The problem of social control and coordination of complex systems in
  sociology: A look at the community cleavage problem.
\newblock {\em IEEE Control Systems Magazine}, 35(3):40--51, 2015.

\bibitem{FJ_1990}
N.E. Friedkin and E.C. Johnsen.
\newblock Social influence and opinions.
\newblock {\em Journal of Mathematical Sociology}, 15:193--206, January 1990.

\bibitem{Cinesi22}
G.~He, Z.~Ci, X.~Wu, and M.~Hu.
\newblock Opinion dynamics with antagonistic relationship and multiple
  interdependent topics.
\newblock {\em IEEE Access}, 10:31595--31606, 2022.

\bibitem{HegselmannKrause}
R.~Hegselmann and U.~Krause.
\newblock Opinion dynamics and bounded confidence: models, analysis and
  simulation.
\newblock {\em J. Artificial Societies and Social Simulation}, 5:1--24, 2002.

\bibitem{hiller}
T.~Hiller.
\newblock Friends and enemies: a model of signed network formation.
\newblock {\em Theoretical Economics}, 3(12):1057--1087, 2017.

\bibitem{HornJohnson}
R.A. Horn and C.R. Johnson.
\newblock {\em Matrix Analysis}.
\newblock Cambridge Univ. Press, Cambridge (GB), 1985.

\bibitem{MeiDorfler}
F.~Liu, S.~Cui, W.~Mei, F.~D\"{o}rfler, and M.~Buss.
\newblock Interplay between homophily-based appraisal dynamics and
  influence-based opinion dynamics: Modeling and analysis.
\newblock {\em IEEE Control Systems Letters}, 5(1):181--186, 2020.

\bibitem{McPherson2001}
M.~McPherson, L.~Smith-Lovin, and J.M. Cook.
\newblock Birds of a feather: Homophily in social networks.
\newblock {\em Annual Review of Sociology}, 27 (1):file 415444, 2001.

\bibitem{HomophilyMei}
W.~Mei, P.~Cisneros-Velarde, G.~Chen, N.E. Friedkin, and F.~Bullo.
\newblock Dynamic social balance and convergent appraisals via homophily and
  influence mechanisms.
\newblock {\em Automatica}, 110:61--67, 2019.

\bibitem{LP-HS-DL:21}
L.~Pan, H.~Shao, and D.~Li.
\newblock Peer selection in opinion dynamics on signed social networks with
  stubborn individuals.
\newblock {\em Neurocomputing}, 477:104--113, 2021.

\bibitem{FJ_extensions_conf}
S.E. Parsegov, A.V. Proskurnikov, R.~Tempo, and N.E. Friedkin.
\newblock A new model of opinion dynamics for social actors with multiple
  interdependent attitudes and prejudices.
\newblock In {\em 54th IEEE Conference on Decision and Control (CDC)}, pages
  3475--3480, 2015.

\bibitem{FJ_extensions}
S.E. Parsegov, A.V. Proskurnikov, R.~Tempo, and N.E. Friedkin.
\newblock Novel multidimensional models of opinion dynamics in social networks.
\newblock {\em IEEE Trans. Aut. Contr.}, 62(5):2270--2285, 2017.

\bibitem{ProskurnikovTempo_1}
A.V. Proskurnikov and R.~Tempo.
\newblock A tutorial on modeling and analysis of dynamic social networks.
  {P}art {I}.
\newblock {\em Annu. Rev. Control.}, 43:65--79, 2017.

\bibitem{ProskurnikovTempo_2}
A.V. Proskurnikov and R.~Tempo.
\newblock A tutorial on modeling and analysis of dynamic social networks.
  {P}art {II}.
\newblock {\em Annu. Rev. Control.}, 45:166--190, 2018.

\bibitem{time_varying_FJ}
A.V. Proskurnikov, R.~Tempo, M.~Cao, and N.~E. Friedkin.
\newblock Opinion evolution in time-varying social influence networks with
  prejudiced agents.
\newblock {\em IFAC-PapersOnLine}, 50(1):11896--11901, 2017.

\bibitem{Rivera2010}
M.T. Rivera, S.B. Soderstrom, and B.~Uzzi.
\newblock Dynamics of dyads in social networks: Assortative, relational, and
  proximity mechanisms.
\newblock {\em Annual Review of Sociology}, 36:file 91115, 2010.

\bibitem{math_analysis}
W.~Rudin.
\newblock {\em Principles of Mathematical Analysis}.
\newblock McGraw Hill, 3rd edition, 1953.

\bibitem{Xia2016StructuralBA}
W.~Xia, M.~Cao, and K.H. Johansson.
\newblock Structural balance and opinion separation in trust-mistrust social
  networks.
\newblock {\em IEEE Transactions on Control of Network Systems}, 3:46--56,
  2016.

\end{thebibliography}
\end{document}